\documentclass{llncs}

\usepackage{makeidx}  
\usepackage{amsmath,amssymb}
\usepackage{multirow}

\DeclareMathOperator{\diag}{diag}

\DeclareMathOperator{\rank}{rank}
\DeclareMathOperator{\nrank}{nrank}

\spnewtheorem{algorithm}{Algorithm}{\bfseries}{\itshape}
\spnewtheorem{fact}{Fact}{\bfseries}{\itshape}
\spnewtheorem{procedure}{Procedure}{\bfseries}{\itshape}
\spnewtheorem{subroutine}{Subroutine}{\bfseries}{\itshape}
\spnewtheorem{flowchart}{Flowchart}{\bfseries}{\itshape}

\begin{document}

\title{\bf Supporting GENP with Random Multipliers}
\author{Victor Y. Pan\inst{1,}\  \inst{2} 
\and Guoliang Qian\inst{2}
 \and Xiaodong Yan\inst{2}\\
}
%
\authorrunning{Victor Y. Pan et al.}   
\institute{Department of Mathematics and Computer Science \\
Lehman College of the City University of New York \\
Bronx, NY 10468 USA \\
\email{victor.pan@lehman.cuny.edu},\\ WWW home page:
\texttt{http://comet.lehman.cuny.edu/vpan/}
\and
Ph.D. Programs in Mathematics and Computer Science \\
The Graduate Center of the City University of New York \\
New York, NY 10036 USA \\
gqian@gc.cuny.edu,
 xyannyc@yahoo.com \\}


\maketitle


\begin{abstract}
We prove that standard  Gaussian random multipliers  
are expected to stabilize numerically  both Gaussian 
elimination with no pivoting and block Gaussian elimination.
Our tests show similar results where we applied
 circulant random
 multipliers instead of Gaussian ones.
\end{abstract}


 


\section{Introduction
}\label{sgenp}


It is well known that even for a nonsingular and well conditioned input matrix,
Gaussian elimination
fails in numerical 
computations with rounding errors
as soon as it encounters 
a vanishing or nearly vanishing
leading (that is northwestern) entry.
In practice the
users avoid such encounters by applying 
GEPP, which stands for Gaussian elimination 
with partial pivoting and has some limited 
formal but ample empirical support. 
Partial pivoting, that is an appropriate row interchange,
 however, takes its toll:
"pivoting usually degrades the performance"
\cite[page 119]{GL96}.
It 
interrupts the stream of arithmetic operations 
with foreign operations of comparison,
involves book-keeping, compromises data locality,
increases communication overhead and data dependence, 
 and
tends to destroy matrix structure.
The users apply {\em Gaussian elimination with no pivoting} (hereafter 
we refer to it as {\em GENP}) or its variants
wherever this application  is numerically safe,
in particular where 
a well  conditioned input matrix
 is
positive definite,
 diagonally 
dominant,  or totally positive.
By following the idea in \cite[Section 12.2]{PGMQ},
briefly revisited in \cite[Section 6]{PQZ13},
we generalize 
this class to all
nonsingular and well
conditioned matrices
by {\em preconditioning} GENP with 
standard Gaussian random multipliers (hereafter we call 
them just {\em Gaussian}). 
We provide formal support for
this approach, which 
 can be extended 
to block Gaussian elimination
(see Section \ref{sbgegenpn}).
Our tests are in good accordance with our formal study.
Moreover multiplication by circulant 
(rather than Gaussian) random multipliers
has the same empirical power, although it
uses by a factor of $n$ fewer random parameters and 
by a factor of $n/\log (n)$ fewer flops
in the case of $n\times n$ inputs.
The latter saving is particularly dramatic in the important case where 
the input matrix has Toeplitz structure.
Our study is
 a new demonstration of the power of randomized matrix algorithms
(cf. \cite{M11}, \cite{HMT11}, and the bibliography therein).


\section{Some definitions and basic results}\label{sdef}


We assume computations in the field $\mathbb R$ of real numbers,
but the extension to the case of the complex field $\mathbb C$
is quite straightforward.
Hereafter ``flop" stands for ``arithmetic operation",
and ``Gaussian matrix" stands for
``standard Gaussian random matrix". 
 The concepts ``large", ``small", ``near", ``closely approximate", 
``ill  conditioned" and ``well conditioned" are 
quantified in the context. By saying ``expect" and ``likely" we
mean ``with probability $1$ or close to $1$"
(we do not use the concept of the expected value).
Next we recall and extend some customary definitions of matrix computations
\cite{GL96}, \cite{S98}.


\subsection{Some basic definitions of matrix computations}\label{smat}

 

$\mathbb  R^{m\times n}$ is the class of real $m\times n$ matrices $A=(a_{i,j})_{i,j}^{m,n}$.
$(B_1~|~\dots~|~B_k)$ is a $1\times k$ block matrix with the blocks $B_1,\dots,B_k$. 
$\diag (B_1,\dots,B_k)=\diag(B_j)_{j=1}^k$ is a $k\times k$ block diagonal matrix 
with the diagonal blocks $B_1,\dots,B_k$. In both cases the blocks $B_j$ can be rectangular.
${\bf e}_i$ is the $i$th coordinate vector of dimension $n$ for
$i=1,\dots,n$. These vectors define
the  $n\times n$ identity
matrix $I_n=({\bf e}_1~|~\dots~|~{\bf e}_n)$. 
$O_{k,l}$ is the $k\times l$ matrix filled with zeros. 
We write $I$ and $O$ where the  matrix size  
 is defined by context. 
$\rank (A)$ denotes the rank of a 
matrix $A$. $A^T$ is its transpose.
 $A_{k,l}$  is its leading, 
that is northwestern  $k\times l$ block 
submatrix, and 
in Section \ref{sbgegenpn}
we also write  $A^{(k)}=A_{k,k}$.
 $A_s^T$ denotes the transpose $(A_s)^T$ 
of a matrix $A_s$, e.g., $A_{k,l}^T$ stands for $(A_{k,l})^T$.
A matrix of a rank $\rho$ has {\em generic rank profile}
if all its leading $i\times i$ blocks are nonsingular for $i=1,\dots,\rho$.
If such a matrix is nonsingular itself, then  it is called {\em strongly nonsingular}.
{\em Preprocessing} $A\rightarrow FA$, $A\rightarrow AH$, and $A\rightarrow FAH$, 
for nonsingular matrices $F$ and $H$, 
reduces
the inversion of a matrix $A$ to
the inversion of the products $FA$, $AH$, or $FAH$,
and similarly for the solution of a linear system of equations.  
\begin{fact}\label{faprepr}
Assume three nonsingular matrices $F$, $A$, and $H$
 and a vector ${\bf b}$.
Then
$A^{-1}=H(FAH)^{-1}F$, 
$FAH{\bf y}=F{\bf b}$, ${\bf x}=H{\bf y}$ if $A{\bf x}={\bf b}$. 
\end{fact}


\subsection{Matrix norms, orthogonality, SVD, and pseudo-inverse}\label{sosvdi}


$||A||=||A||_2=\sup_{{\bf v}^T{\bf v}=1}||A{\bf v}||$
is the spectral norm of a matrix $A=(a_{i,j})_{i,j=1}^{m,n}$,
$||AB||\le ||A||~||B||$. 
A real matrix $Q$ is called  
{\em orthogonal} if $Q^TQ=I$ 
 or $QQ^T=I$. 

An {\em SVD} or {\em full SVD} of an $m\times n$ matrix $A$ of a rank 
 $\rho$ is a factorization
\begin{equation}\label{eqsvd}
A=S_A\Sigma_AT_A^T
\end{equation}
where $S_A=({\bf s}_i)_{i=1}^m$ and $T_A=({\bf t}_j)_{j=1}^n$ are square orthogonal matrices, 
$\Sigma_A=\diag(\widehat \Sigma_A,O_{m-\rho,n-\rho})$, 
$\widehat \Sigma_A=\diag(\sigma_j(A))_{j=1}^{\rho}$,
$\sigma_j=\sigma_j(A)=\sigma_j(A^T)$
is the $j$th largest singular value of a matrix $A$, and 
$\sigma_j=0$ for $j>\rho$,
$\sigma_{\rho}>0$,  
 $\sigma_1=\max_{||{\bf x}||=1}||A{\bf x}||=||A||$,
and 
\begin{equation}\label{eqmins}
\min_{\rank(B)\le s-1}~~~||A-B||=\sigma_{s}(A),~s=1,2,\dots,
\end{equation}

 
\begin{fact}\label{faccondsub} 
If $A_0$ is a 
submatrix of a 
matrix $A$, 
then
$\sigma_{j} (A)\ge \sigma_{j} (A_0)$ for all $j$.
\end{fact} 

 
\begin{proof}
Extend \cite[Corollary 8.6.3]{GL96}.
\end{proof}

 
\begin{fact}\label{faintpr} (Cf. \cite[Corollary 8.6.3]{GL96}.)  
Suppose $r+l\le n\le m$, $l\ge 0$,
 $1\le k\le r$,
 $A\in \mathbb R^{m\times n}$,
and
$A_{m,r}$ is the leftmost $m\times r$ block of the matrix $A$. 
Then $\sigma_{k} (A_{m,r})\ge \sigma_{k+l} (A_{m,r+l})$.
\end{fact}

$A^+=T_A\diag(\widehat \Sigma_A^{-1},O_{n-\rho,m-\rho})S_A^T$ is the Moore--Penrose 
pseudo-inverse of the matrix $A$ of (\ref{eqsvd}). 
  $A^{+T}$ stands for $(A^+)^T=(A^T)^+$,
and $A_s^+$ stands for $(A_s)^{+}$, e.g., $A_{k,l}^+$ denotes $(A_{k,l})^+$. 

If a matrix $A$ has full column rank $\rho$, then
 $A^+=(A^TA)^{-1}A^T$ and
\begin{equation}\label{eqnrm+}
||A^+||=1/\sigma_{\rho}(A).
\end{equation}



\begin{corollary}\label{cointpr}  
Assume that $\rank(A_{m,r})=r$ and $\rank(A_{m,r+l})=r+l$
for the matrices $A_{m,r}$  and $A_{m,r+l}$ of Fact \ref{faintpr}.
Then $||A_{m,r}^+||\le ||A_{m,r+l}^+||$.
\end{corollary}
\begin{proof}
Combine Fact  \ref{faintpr} for $k=r$ with equation (\ref{eqnrm+}).
\end{proof}


\begin{theorem}\label{thpert}  \cite[Corollary 1.4.19]{S98}.
Assume a pair of square matrices $A$ (nonsingular) and $E$ such that 
$||A^{-1}E||\le 1$. Then $||(A+E)^{-1}||\le\frac{||A^{-1}||}{1-||A^{-1}E||}$ and 
moreover $\frac{||(A+E)^{-1}-A^{-1}||}{||A^{-1}||}\le \frac{||A^{-1}||}{1-||A^{-1}E||}$.
\end{theorem}


\subsection{Condition number, numerical rank and  generic conditioning profile
}\label{scnpn}


$\kappa (A)=\frac{\sigma_1(A)}{\sigma_{\rho}(A)}=||A||~||A^+||$ is the condition 
number of an $m\times n$ matrix $A$ of a rank $\rho$. Such matrix is 
{\em ill conditioned}
if the ratio $\sigma_1(A)/\sigma_{\rho}(A)$
is large. If the ratio is reasonably bounded,
then the matrix is {\em well conditioned}.  
An $m\times n$ matrix $A$ has a {\em numerical rank} 
$r=\nrank(A)\le \rho=\rank (A)$ 
if the ratios $\sigma_{j}(A)/||A||$
are small for $j>r$ but not for $j\le r$. 


\begin{remark}\label{renul}
One can specify the adjective ``small" 
above as
``smaller than  a fixed positive tolerance"
and similarly specify ``closely" and ``well conditioned".
The specification can be a challenge,
e.g., for the matrix $\diag(1.1^{-j})_{j=0}^{999}$.
\end{remark}

If a well conditioned  $m\times n$ matrix $A$ has a 
rank $\rho<l=\min\{m,n\}$, 
 then all its close neighbors 
have 
numerical rank $\rho$ and  almost
all of them have rank $l$.
Conversely, if a matrix $A$
has a positive
numerical rank $r$,
then the $r$-truncation $A_{r}$,
obtained by setting to 0 
all singular values $\sigma_{j}(A)$ for $j>r$, is
a  well conditioned  rank-$r$ approximation to the matrix $A$
within the error norm bound $\sigma_{r+1}(A)$ (cf. (\ref{eqmins})).
It follows that a matrix is ill conditioned if and only if 
it is close to a matrix having a smaller rank
and
a matrix has a numerical rank $r$
if and only if it can be closely approximated by a well conditioned
matrix having full rank $r$.
Rank-revealing factorizations  of a matrix $A$ that  has a small numerical rank $r$,
but possibly has  a large rank $\rho$,
 produce its rank-$r$ approximations  
at a lower computational cost
 \cite{P00}. 
The
 randomized algorithms
of \cite{HMT11} decrease the computational cost
further. 
An $m\times n$ matrix 
has {\em generic conditioning profile}
if it
has a numerical rank $r$ and if
its leading $i\times i$ blocks are nonsingular and well conditioned for $i=1,\dots,r$.
Such matrix is {\em strongly well conditioned}
if it has full numerical rank  $r=\min\{m,n\}$.


\subsection{Block Gaussian elimination and GENP
}\label{sbgegenpn}


For a 
nonsingular
$2\times 2$ block matrix $A=\begin{pmatrix}
B  &  C  \\
D  &  E
\end{pmatrix}$
with a nonsingular {\em pivot block} $B=A^{(k)}$, 
 define 
$S=S(A^{(k)},A)=E-DB^{-1}C$,
the {\em Schur complement} of $A^{(k)}$ in $A$,
and
the block factorization,

\begin{equation}\label{eqgenp}
\begin{aligned}
A=\begin{pmatrix}
I_k  &  O_{k,r}  \\
DB^{-1}  & I_r
\end{pmatrix}
\begin{pmatrix}
B  &  O_{k,r} \\
O_{r,k}  &  S
\end{pmatrix}
\begin{pmatrix}
I_k  &  B^{-1}C  \\
O_{k,r}  & I_r
\end{pmatrix}
\end{aligned}.
\end{equation} 

Apply this factorization recursively  
to  the pivot block
$B$ and its Schur complement $S$
 and arrive at
 the block Gaussian elimination process,
completely defined 
by the sizes of the pivot blocks.
The recursive process 
 either fails, where its pivot block turns out to be singular,
in particular where it is a vanishing pivot entry of GENP,
or can continue until all pivot blocks become 
nonzero scalars. When this occurs  we arrive at GENP.
Factorization (\ref{eqgenp}) defines the block elimination of the
first $k$
columns of the matrix $A$, 
whereas $S=S(A^{(k)},A)$
is the matrix produced at this elimination step.
Now assume that the pivot dimensions 
$d_1,\dots,d_r$ and $\bar d_1,\dots,\bar d_{\bar r}$ of 
two block elimination processes
 sum to the same integer $k$, that is 
$k=d_1+\dots+d_r=\bar d_1+\dots+\bar d_{\bar r}$.
Then verify that 
both processes produce the same Schur complement $S=S(A^{(k)},A)$.


\begin{theorem}\label{thsch}
In every step of the recursive block factorization process based on (\ref{eqgenp})  
 every diagonal block of a block diagonal factor is either 
a leading block of the input matrix $A$ or the Schur complement $S(A^{(h)},A^{(k)})$
 for some integers $h$ and $k$ such that $0<h<k\le n$ and
$S(A^{(h)},A^{(k)})=(S(A^{(h)},A))^{(h)}$.
\end{theorem}


\begin{corollary}\label{corec}
The recursive block factorization process based on equation (\ref{eqgenp})  
can be completed by involving no vanishing pivot elements and no singular
pivot blocks
if and only if the input matrix $A$ has 
generic rank profile.
\end{corollary}
\begin{proof}
Combine Theorem \ref{thsch} with equation 
$\det A=(\det B)\det S$,
implied by (\ref{eqgenp}).
\end{proof}

The following result shows that 
the  pivot blocks are not close to singular matrices
where block Gaussian elimination (e.g., GENP) is applied 
to a strongly 
well conditioned input matrix.


\begin{theorem}\label{thnorms} (Cf. \cite[Theorem 5.1]{PQZ13}.)
Assume GENP
or block Gaussian elimination
applied to 
 an
$n\times n$
matrix $A$ and
write $N=||A||$ and $N_-=\max_{j=1}^n ||(A^{(j)})^{-1}||$,
and so $N_-N\ge ||A||~||A^{-1}||\ge 1$. 
Then the absolute values of all pivot elements of GENP 
and the norms of all pivot blocks of 
block Gaussian elimination
do not exceed $N_+=N+N_-N^2$,
whereas the absolute values of the reciprocals of these 
elements and the norms of the inverses of the blocks do not
exceed $N_-$.
\end{theorem}
\begin{proof}
Observe that the inverse $S^{-1}$ of the Schur complement $S$
in (\ref{eqgenp}) is the southeastern block of the inverse $A^{-1}$
and obtain
$||B||\le N$, $||B^{-1}||\le N_-$, and 
$||S^{-1}||\le ||A^{-1}||\le N_-$. Moreover
$||S||\le N+N_-N^2$, due to (\ref{eqgenp}). Now the claimed bound follows 
from Theorem \ref{thsch}.
\end{proof}
\noindent Invert (\ref{eqgenp}) to obtain 
$\begin{aligned}
A^{-1}=\begin{pmatrix}
I_k  &  -B^{-1}C  \\
O_{k,r}  & I_r
\end{pmatrix}
\begin{pmatrix}
B^{-1}  &  O_{k,r} \\
O_{r,k}  &  S^{-1}
\end{pmatrix}
\begin{pmatrix}
I_k  &  O_{k,r}  \\
-DB^{-1}  & I_r
\end{pmatrix}
\end{aligned}$,
 extend this factorization  recursively,
apply it to the inversion of the matrix $A$ and the solution of a linear system $A{\bf x}={\bf b}$,
and extend the above analysis.
 

\begin{remark}\label{reir}
For a strongly nonsingular input matrix $A$
 block factorization (\ref{eqgenp}) can be extended 
to computing the complete recursive factorization, 
which  defines GENP.  By virtue of Theorem \ref{thnorms} 
the norms of the inverses of
all pivot blocks involved in this computation
are at most $N_-$.
If the matrix $A$ is also strongly well conditioned, then
 we have a reasonable upper bound on $N_-$,
and  so
in view of Theorem \ref{thpert}
the inversion of all pivot blocks is numerically
safe. In this case we say that {\em GENP and block Gaussian elimination are locally safe} for the matrix $A$.
In a sense locally safe GENP and
 block Gaussian elimination are at least as safe numerically as
GEPP because a locally safe recursive factorization involves 
neither divisions by absolutely small pivot entries 
(recall that  pivoting has been introduced  
precisely in order to avoid such divisions) 
 nor inversions of ill conditioned pivot blocks. 
 Let us also compare the
 magnification of
the perturbation norm bounds of  Theorem \ref{thpert} in GEPP and
in  the process of recursive factorization, which defines GENP
and block Gaussian elimination.
We observe immediately that in the recursive factorization 
only the factors of the leading blocks  
and the Schur complements
can contribute to this magnification, namely at most $\log_2(n)$
such factors can contribute to the norm of each of the output triangular or block triangular
factors $L$ and $U$. This implies the upper bound $(N_+N_-)^{\log_2(n)}$
on their norms, which can be compared favorably to the sharp upper bound on the growth factor
$2^{n-1}$ for GEPP \cite[page 119]{GL96} and \cite[Theorem 3.4.12]{S98}.
\end{remark}


\section{Singular values of the  matrix products (deterministic estimates)
and GENP with preprocessing}\label{smrc}


Fact \ref{faprepr} reduces the tasks of inverting a nonsingular and well
conditioned matrix $A$
and solving a linear system  $A{\bf x}={\bf b}$ to similar tasks for the
matrix $FAH$ and multipliers $F$ and $H$ of our choice. 
 Remark \ref{reir} motivates the choice for which the 
matrix $FAH$ is strongly nonsingular 
and strongly well conditioned. In Section \ref{ssgnp}
we 
prove that this is likely to occur already where one of the multipliers
$F$ and $H$
 is the identity matrix $I$
 and another one is Gaussian matrix, 
and therefore also
where both $F$ an $H$ are independent Gaussian matrices.
In this section we prepare  background for that proof by 
estimating the norms of the matrices $(FA)_{k,k}=F_{k,m}A_{m,k}$
and $(AH)_{k,k}=A_{k,n}H_{n,k}$ and of their inverses
for general (possibly 
nonrandom) multipliers
$F$ and $H$. 
We will keep writing $M_{s}^T$ and $M_{s}^+$  for $(M_{s})^T$ and $(M_{s})^+$,
respectively,
where $M$ can stand for  $A$, $F$, or $H$, say, 
and where $s$ can be any subscript such as a pair $(k,l)$ or $(A,r)$.
We begin with  two simple lemmas.


\begin{lemma}\label{lepr2} 
Suppose $S$ and $T$ 
are square orthogonal matrices. Then 
$\sigma_{j}(SA)=\sigma_j(AT)=\sigma_j(A)$ for all $j$.
\end{lemma}


\begin{lemma}\label{lepr1} 
Suppose $\Sigma=\diag(\sigma_i)_{i=1}^{n}$, $\sigma_1\ge \sigma_2\ge \cdots \ge \sigma_n$,
$F\in \mathbb R^{r\times n}$, $H\in \mathbb R^{n\times r}$.
Then 
$\sigma_{j}(F\Sigma)\ge\sigma_{j}(F)\sigma_n$ and
$\sigma_{j}(\Sigma H)\ge\sigma_{j}(H)\sigma_n$ for all $j$.
If also $\sigma_n>0$, then 
 $\rank (F\Sigma)=\rank (F)$ and $\rank (\Sigma H)=\rank (H)$.
\end{lemma}


\noindent The following theorem bounds the norms $||(FA)^+||$ and $||(AH)^+||$
for three matrices $A$, $F$ and $H$.

\begin{theorem}\label{1}
Suppose $A\in \mathbb R^{m\times n}$,
$F\in\mathbb R^{r\times m}$, $H\in\mathbb R^{n\times r}$,
 $r\le \rho$
for 
$\rho=\rank (A)$,  
 $A=S_A\Sigma_AT^T_A$ (cf. (\ref{eqsvd})), 
 $\widehat F=FS_A$, and
$\widehat H=T_A^TH$.
 Then 
\begin{equation}\label{eqfa}
\sigma_j(FA)\ge \sigma_k(A) ~\sigma_j(\widehat F_{r,k})~{\rm for~all}~k\le m~{\rm and~all}~j,
\end{equation}
\begin{equation}\label{eqah}
\sigma_j(AH)\ge \sigma_l(A) ~\sigma_j(\widehat H_{l,r})~{\rm for~all}~l\le n~{\rm and~all}~j.
\end{equation}
\end{theorem}


\begin{proof}
Note that
 $AH=S_A\Sigma_AT_A^TH$, and so  
$\sigma_j(AH)=\sigma_j(\Sigma_AT_A^TH)=
\sigma_j(\Sigma_A\widehat H)$ for all $j$ 
by virtue of Lemma \ref{lepr2}, because 
$S_A$ is a square orthogonal matrix. 
Furthermore it follows from Fact \ref{faccondsub} that
$\sigma_j(\Sigma_A\widehat H)\ge 
\sigma_j(\Sigma_{l,A}\widehat H_{l,r})$
for all $l\le n$. Combine this bound with the latter equations,
then apply Lemma \ref{lepr1}, and obtain bound (\ref{eqah}).
Similarly deduce bound (\ref{eqfa}). 
\end{proof}


\begin{corollary}\label{cor1}
Keep the assumptions and definitions of Theorem \ref{1}. Then
 
(i)  $\sigma_{r}(AH)\ge 
\sigma_{\rho}(A)\sigma_{r}(\widehat H_{\rho,r})=
\sigma_{r}(\widehat H_{\rho,r})/||A^+||$,

(ii) $||(AH)^+||\le ||A^+||~||\widehat H_{\rho,r}^+||$
if $\rank(AH)=\rank(\widehat H_{\rho,r})=r$, 

(iii) $\sigma_{r}(FA)\ge 
\sigma_{\rho}(A)\sigma_{r}(\widehat F_{r,\rho})=
\sigma_{r}(\widehat F_{r,\rho})/||A^+||$, and

(iv) $||(FA)^+||\le ||A^+||~||\widehat F_{r,\rho}^+||$
if $\rank(FA)=\rank(\widehat F_{r,\rho})=r$.
\end{corollary}


\begin{proof}
Substitute $j=r$ and $l=\rho$
into bound (\ref{eqah}), recall
(\ref{eqnrm+}), and
obtain part (i).
If $\rank(AH)=\rank(\widehat H_{l,r})=r$,
then apply (\ref{eqnrm+})
to obtain that 
$\sigma_{r}(AH)=1/||(AH)^+||$ and 
$\sigma_{r}(\widehat H_{l,r})=1/||\widehat H_{l,r}^+||$. 
Substitute these equations 
into part (i) and obtain part (ii).
Similarly prove parts (iii) and (iv).
\end{proof}


Let us  extends Theorem \ref{1}  to the leading
 blocks of the matrix products.


\begin{corollary}\label{cogh}
Keep the assumptions and definitions of Theorem \ref{1} and fix
two positive integers $k$ and $l$
such that
$k\le m$,  $l\le n$. 
Then  
(i) $||(FA)_{k,l}^+||\le ||\widehat F_{k,m}^+||~||A_{m,l}^+||\le ||\widehat F_{k,m}^+||~||A^+||$
 if $m\ge n=\rho$ and if the matrices $(FA)_{k,l}$ 
and  $\widehat F_{k,m}$ have  full rank, whereas
(ii) $||(AH)_{k,l}^+||\le ||\widehat H_{n,l}^+||~||A_{k,n}^+||\le
||\widehat H_{n,l}^+||~||A^+||$ if $m=\rho\le n$
and if the matrices  $(AH)_{k,l}$ and $\widehat H_{n,l}$ 
have full rank. 
\end{corollary} 


\begin{proof}
Recall that 
$(FA)_{k,l}=F_{k,m}A_{m,l}$
and the matrix $A_{m,l}$ has full rank if $m\ge n=\rho$.
Apply Corollary \ref{cor1}  for $A$ and $F$ replaced by 
$A_{m,l}$ and $F_{k,m}$, respectively,
and obtain that $||(FA)_{k,l}^+||\le ||\widehat F_{k,m}^+||~||A_{m,l}^+||$.
Combine (\ref{eqnrm+}) and Corollary \ref{cointpr} and deduce that 
$||A_{m,l}^+||\le||A^+||$. Combine  the two latter 
inequalities 
to complete the proof of part (i). 
Similarly prove part (ii).
\end{proof}


The following definition formalizes the assumptions of 
Corollaries \ref{cor1} and  \ref{cogh}.

\begin{definition}\label{def1}
Assume the matrices $A$, $F$, $\widehat F$, $H$, and  $\widehat H$ of Theorem \ref{1}.
Then the matrix pair $(A,H)$ (resp. $(F,A)$) has {\em full rank} 
if the matrices $AH$, $\widehat H_{\rho,r}$ (resp. $FA$, $\widehat F_{r,\rho})$)
and $A$ have full rank.
This pair has full rank and is {\em well conditioned} 
if in addition the matrices  $\widehat H_{\rho,r}$ (resp. $\widehat F_{r,\rho})$)
and $A$ are well conditioned,
whereas it has {\em generic rank profile} if  $\rank(A)=\rho$ and
$\rank((AH)_{k,k})=\rank(\widehat H_{\rho,k})=k$ 
(resp. $\rank((FA)_{k,k})=\rank(\widehat F_{k,\rho})=k$) for $k=1,\dots,r$.
 The pair has generic rank profile
and  is {\em strongly well conditioned} if
in addition 
the matrices $\widehat H_{\rho,k}$  
(resp. $\widehat F_{k,\rho}$) for $k=1,\dots,r$
 are well conditioned.
\end{definition}

\begin{remark}\label{regenppr}
Fact \ref{faprepr}, 
Corollary \ref{corec} and Theorem \ref{thnorms} together imply the
following guiding rule.
 Suppose  
$A\in \mathbb R^{m\times n}$,
$F\in\mathbb R^{r\times m}$, $H\in\mathbb R^{n\times r}$,
 $r\le \rho=\rank (A)$, and
 the matrix pair $(A,H)$ for $m\le n$ or $(F,A)$
 for $m\ge n$
has generic rank profile and 
 is  strongly well conditioned. Then
 GENP is locally safe for the
matrix products $AH$ or $FA$, respectively
 (see Remark \ref{reir} on the concept ``locally safe").
\end{remark}


\section{Benefits of using random matrix multipliers 
}\label{sapsr1}


Next we 
define Gaussian matrices and recall their basic properties.
In Section \ref{ssgnp}  we show
that these multipliers are expected to support 
locally safe GENP, and
in Section \ref{srnd} comment on using non-Gaussian random
multipliers. 

\subsection{Gaussian 
matrix, its rank, norm and condition estimates}\label{srrm}




\begin{definition}\label{defrndm}
A matrix is  {\em standard Gaussian random} or just
{\em Gaussian} if it is filled with i.i.d.
Gaussian random
variables  having mean $0$ and variance $1$. 
\end{definition}


\begin{fact}\label{facdgr}
A Gaussian matrix  is rank deficient with probability 0.
\end{fact}
\begin{proof}
Assume a rank deficient $m\times n$ matrix where $m\ge n$, say.
Then the determinants 
of all its $n\times n$ submatrices vanish. This implies 
$\begin{pmatrix} m \\n \end{pmatrix}$
polynomial
equations on the entries, that is the rank deficient matrices form
an algebraic variety of a lower dimension in the linear space 
$\mathbb R^{m\times n}$. 
($V$ is an algebraic variety of a dimension $d\le N$
in the space $\mathbb  R^{N}$ if it is defined by $N-d$
polynomial equations and cannot be defined by fewer equations.)
Clearly such a variety has Lebesgue (uniform) and
Gaussian measure 0, both being absolutely continuous
with respect to one another.
\end{proof}


\begin{corollary}\label{codgr} 
 A Gaussian 
matrix has generic rank profile 
with probability 1.
\end{corollary}
Hereafter $\nu_{j,m,n}$ denote the random variables
$\sigma_j(G)$ for Gaussian $m\times n$ matrix $G$
and all $j$, whereas
$\nu_{m,n}$, $\nu_{m,n}^+$, and $\kappa_{m,n}$ denote the 
 random variables
 $||G||$,
 $||G^+||$, and $\kappa(G)=||G||~||G^+||$, respectively.
Note that $\nu_{j,n,m}=\nu_{j,m,n}$, $\nu_{n,m}=\nu_{m,n}$,
$\nu_{n,m}^+=\nu_{m,n}^+$, and $\kappa_{n,m}=\kappa_{m,n}$.

\begin{theorem}\label{thsignorm}
(Cf. \cite[Theorem II.7]{DS01}.)
Suppose 
$h=\max\{m,n\}$, $t\ge 0$,  and
$z\ge 2\sqrt {h}$. 
Then  

Probability$\{\nu_{m,n}>z\}\le
\exp(-(z-2\sqrt {h})^2/2\}$ and 

Probability$\{\nu_{m,n}>t+\sqrt m+\sqrt n\}\le
\exp(-t^2/2)$.
\end{theorem}

\begin{theorem}\label{thsiguna} 
Suppose 
$m\ge n$, and $x>0$ and write $\Gamma(x)=
\int_0^{\infty}\exp(-t)t^{x-1}dt$ 
and $\zeta(t)=
t^{m-1}m^{m/2}2^{(2-m)/2}\exp(-mt^2/2)/\Gamma(m/2)$. 
Then 

(i) Probability $\{\nu_{m,n}^+\ge m/x^2\}<\frac{x^{m-n+1}}{\Gamma(m-n+2)}$
for $n\ge 2$ and

(ii)  Probability $\{\nu_{m,1}^+\ge x\}\le
(m/2)^{(m-2)/2}/(\Gamma(m/2)x^m)$.
\end{theorem}
\begin{proof}
(i) See \cite[Proof of Lemma 4.1]{CD05}.
(ii) $G\in \mathbb R^{m\times 1}$ is a vector of length $m$.
 So, with probability 1,
$G\neq 0$, 
$\rank (G)=1$, $||G^+||=1/||G||$,
 and consequently
Probability $\{||G^+||\ge x\}=$
 Probability $\{||G||\le 1/x\}=\int_{0}^{1/x} \zeta(t) dt$.
Note that $\exp(-mt^2/2)\le 1$, and so
$\int_{0}^{1/x} \zeta(t) dt< c_m\int_{0}^{1/x}t^{m-1}dt=c_m/(mx^m)$ 
where $c_m=m^{m/2}2^{(2-m)/2}/\Gamma(m/2)$.
\end{proof}


The following condition estimates
 from \cite[Theorem 4.5]{CD05}
are quite tight for large values $x$, but for $n\ge 2$ even tighter
estimates (although more involved)
 can be found in \cite{ES05}. (See \cite{D88} and  \cite{E88}
on the early study.)


\begin{theorem}\label{thmsiguna}   
 With probability 1, we have
$\kappa_{m,1}=1$. If $m\ge n\ge 2$, then
 
Probability$\{\kappa_{m,n}m/(m-n+1)>x\}\le \frac{1}{2\pi}(6.414/x)^{m-n+1}$
for $x\ge m-n+1$.
\end{theorem}


\begin{corollary}\label{cogfrwc} 
 A Gaussian 
matrix has generic rank profile 
with probability 1
and is expected to be well conditioned.
\end{corollary}
\begin{proof}
Combine Corollary \ref{codgr} and Theorem \ref{thmsiguna}.
\end{proof}


 

\subsection{Supporting  GENP with Gaussian multipliers
}\label{ssgnp}

\begin{lemma}\label{lepr3}
Suppose $H$ is Gaussian matrix, 
$S$ and $T$ are orthogonal matrices, $H\in \mathbb R^{m\times n}$,
$S\in \mathbb R^{k\times m}$, and $T\in \mathbb R^{n\times k}$
for some $k$, $m$, and $n$.
Then $SH$ and $HT$ are Gaussian matrices.
\end{lemma}


\begin{theorem}\label{thrkpg}
Suppose $A\in \mathbb R^{m\times n}$, $F\in \mathbb R^{r\times m}$,
$H\in \mathbb R^{m\times r}$, $F$ and $H$ are Gaussian matrices, 
and $\rank (A)=\rho$.
Then $\rank (FA)=\rank(AH)=\min \{r,\rho\}$ with probability 1. 
\end{theorem}


\begin{proof}
Suppose $A=S_{A}\Sigma_{A}T_{A}^T$ is SVD of (\ref{eqsvd}).
Then $FA=FS_{A}\Sigma_{A}T_{A}^T=G\Sigma_{A}T_{A}^T$
where $G=FS_A$ is Gaussian $r\times m$ matrix by virtue of 
Lemma \ref{lepr3}.
Clearly $\rank(FA)=\rank(G\Sigma_{A}T_{A}^T)=\rank(G\Sigma_{A})$
because $T_{A}$ is a square orthogonal matrix.
Moreover $\rank(G\Sigma_{A})=\rank(GD_{\rho})$ 
where $D_{\rho}=\diag(I_{\rho},O_{m-\rho,n-\rho}))$, and so
$GD_{\rho}$ is Gaussian  $r\times \rho$ matrix because it is a submatrix of 
the Gaussian matrix $G$.  
Therefore $\rank(FA)=\rank(GD_{\rho})$ is equal to $\min\{r,\rho\}$
  with probability 1  by virtue of
Fact \ref{facdgr}. Similarly obtain that $\rank(AH)=\min\{r,\rho\}$
with probability 1. 
\end{proof}


\begin{corollary}\label{cor10}
Keep the assumptions and definitions of Theorem \ref{1}. 
Suppose the matrix $A$ has full rank $\rho=\min\{m,n\}$, 
 $k\le r\le\rho$,
and $F=FS_A$
and $H=T_A^TH$ are Gaussian matrices. 
Then (i) so
are the matrices $\widehat F$,
$\widehat H$ and 
 all their submatrices, in particular 
$\widehat F_{k,\rho}$ and 
 $\widehat H_{\rho,k}$, and
 (ii) with probability $1$,  
$\rank((AH)_{k,k})=k$ if $m\le n$, $\rank((FA)_{k,k})=k$ if $m\le n$, and
$\rank(\widehat H_{\rho,k})=\rank(\widehat F_{k,\rho})=k$.
\end{corollary}


\begin{proof}
If $H$ and $F$ are Gaussian matrices, then so are the matrices
$\widehat H$ and $\widehat F$ by virtue of Lemma \ref{lepr3}.
Consequently so are all their submatrices. This proves
 parts (i) and  by virtue of Fact \ref{facdgr} also implies the
 equations $\rank(\widehat H_{\rho,k})=\rank(\widehat F_{k,\rho})=k$ of part (ii). 
Now recall that $(AH)_{k,k}=A_{k,n}H_{n,k}$,
and hence $\rank ((AH)_{k,k})=\rank(A_{k,n}H_{n,k})$.
This is equal to $\rank(A_{k,n})$ with probability 1 
by virtue of Theorem \ref{thrkpg} because
$H_{n,k}$ is a Gaussian matrix and because $k\le \rho\le n$.
Finally obtain that 
 $\rank(A_{k,n})=k$ for $k\le \rho=m$, and so 
$\rank((AH)_{k,k})=k$.
Similarly prove that $\rank ((FA)_{k,k})=k$ for $k\le \rho=n$.
\end{proof}


\begin{corollary}\label{cogmforalg}
The  choice of 
Gaussian multipliers 
$F$ where $m\le n$ or $H$ where $m\ge n$ is expected 
to satisfy the assumptions of Remark \ref{regenppr}
(thus supporting application of 
GENP
to the matrix $FA$ where $m\le n$ or $AH$ 
where $m\ge n$)
provided that the $m\times n$ matrix $A$
is nonsingular and well conditioned.
\end{corollary}
\begin{proof}
Combine Corollaries \ref{cogfrwc}
and \ref{cor10}.
\end{proof}


\subsection{Structured random multipliers}\label{srnd}


Given $n\times n$ matrices $A$, 
$F$
and  $H$, we need $2n^3-n^2$ flops 
to compute each of the products $FA$ and $AH$,
but we only need order of $n^2\log (n)$ flops
 to compute such products where
 $F$ and $H$ are circulant matrices (cf. \cite{p01}).
Furthermore we need just $n$ random parameters to define 
a Gaussian circulant  $n\times n$ matrix $C=(c_{i-j\mod n})_{i,j=0}^{n-1}$,
expected to be very well conditioned \cite{PQa}.
Can we extend our results to these random  multipliers? 
The proof and consequently the claim of Fact \ref{facdgr} 
can be immediately extended,
and so we can still
satisfy 
 with probability 1
the generic rank profile assumption
of  Remark \ref{regenppr}.
Moreover we can satisfy that assumption
with probability close to 1 even where we 
fill the multipliers $F$ and $H$ 
with i.i.d. random variables defined
under the discrete uniform probability distribution
over a fixed sufficiently large finite set
(see Appendix \ref{srsnrm}).
We cannot extend our proof that the pair  
$(T_A,H)$ is strongly well conditioned, however,
because we cannot extend Lemma \ref{lepr3} to this case.
Nevertheless empirically circulant random multipliers
turn out to support GENP
 quite strongly
(see our next section and also compare \cite{HMT11}, and \cite{M11}
on randomized low-rank approximation of a matrix). 
By engaging simultaneously 
 two independent structured random multipliers $F$ and $H$
we may enhance their power.
 

\section{Numerical Experiments}\label{sexp}

 
Our numerical experiments with general, Hankel, Toeplitz and circulant 
random  matrices 
have been performed in the Graduate Center of the City University of New York 
on a Dell server with a dual core 1.86 GHz
Xeon processor and 2G memory running Windows Server 2003 R2. The test
Fortran code has been compiled with the GNU gfortran compiler within the Cygwin
environment.  Random numbers have been generated 
assuming the 
standard Gaussian 
probability distribution. 
 The tests have been designed by the first author 
and performed by his coauthors.

Tables \ref{tab44g} and  \ref{tab44} show the results of our tests of the solution 
of  nonsingular well conditioned linear systems $A{\bf y}={\bf b}$ of 
$n$ equations with random normalized vectors ${\bf b}$
whose coefficient matrices had 
singular
$n/2\times n/2$   leading principal blocks for $n=64, 256,1024$.
Namely by following
\cite[Section 28.3]{H02}
we generated the  $n\times n$ 
input matrices 
$A=\begin{pmatrix}
A_k  &  B  \\
C    &  D
\end{pmatrix}$
where
 $A_k$ was a $k\times k$  matrix,
$B$, $C$ and $D$ were Toeplitz random  matrices such that 
 $||B||\approx ||C||\approx ||D||\approx ||A_k||\approx 1$,
$n=2^s$, $s=6,7,8, 10$,
and $A_k=U\Sigma V^T$ for $k=n/2$, $\Sigma=\diag (\sigma_i)_{i=1}^k$,
$\sigma_i=1$ for $i=1,\dots,k-h$, $\sigma_i=0$ for $i=k-h+1,\dots,k$, $h=4$,
and $U$ and $V$ were $k\times k$ orthonormal random  matrices, computed as 
the unique $k\times k$ factors $Q(V)$ in the QR factorization of $k\times k$ random matrices $V$. 
We have performed 100 numerical tests
 for each dimension $n$ and have computed 
the maximum, minimum and average relative residual norms 
$||A{\bf y}-{\bf b}||/||{\bf b}||$ as well as the standard deviation.

In our tests the norms $||A^{-1}||$ ranged from 
70 to $4\times 10^6$ (see Table \ref{tabinorm}),
and so   GEPP was expected to output accurate solutions 
to the linear systems $A{\bf y}={\bf b}$, and indeed we always observed this 
 in our tests
 (see Table \ref{tab44gpp}).  
GENP, however, was expected to fail for these systems, 
because the leading block $A_k$ of the matrix $A$  was singular, having nullity 
$k-\rank (A_k)=4$.
Indeed  this has caused poor performance of GENP in our tests,
which
has consistently 
output corrupted solutions,   with the residual norms
ranging from 10 to $10^8$. 
In view of Remark \ref{cogmforalg} 
we expected to fix this deficiency by means of multiplication
by Gaussian random matrices, and indeed  in our tests
 we observed consistently
the residual norms below
 $4\times 10^{-9}$ for all inputs  
(see Table \ref{tab44g}). Furthermore the tests 
showed the same power of preconditioning 
where we used the Gaussian circulant multipliers (see Table \ref{tab44}). 
As could be expected 
the output accuracy of GENP 
preprocessed with nonunitary  random multipliers  
tended to deteriorate a little versus GEPP
 in our tests,
but the output residual norms were small enough to support
application of the inexpensive iterative refinement,
whose single step decreased the output residual norm by  factors of $10^h$ for $h>1$
in the case of Gaussian multipliers and $h\ge 2$ in the case of Gaussian circulant multipliers
(see Tables \ref{tab44g} and  \ref{tab44} and Remark \ref{reir}).


   

\begin{table}[ht]
  \caption{the norms $||A||^{-1}$ for the input matrices $A$}
  \label{tabinorm}
  \begin{center}
    \begin{tabular}{| c | c | c | c | c | c |}
      \hline
      \bf{dimension} &  \bf{min} & \bf{max} & \bf{mean} & \bf{std} \\ \hline
 $64$ &  $6.9\times 10^1$ & $2.3\times 10^3$ & $4.6\times 10^4$ & $6.4\times 10^3$ \\ \hline
 $256$ &  $8.4\times 10^2$ & $1.1\times 10^4$ & $5.8\times 10^5$ & $ 5.8\times 10^4$ \\ \hline
 $1024$ & $3.5\times 10^3$ & $9.9\times 10^4$ & $4.0\times 10^6$ & $4.3\times 10^5$ \\ \hline
    \end{tabular}
  \end{center}
\end{table}


\begin{table}[ht]
  \caption{Relative residual norms  of GEPP}
  \label{tab44gpp}
  \begin{center}
    \begin{tabular}{| c | c | c | c | c | c |}
      \hline
      \bf{dimension}  & \bf{min} & \bf{max} & \bf{mean} & \bf{std} \\ \hline

 $64$ &  $2.0\times 10^{-15}$ & $6.9\times 10^{-13}$ & $3.2\times 10^{-14}$ & $8.9\times 10^{-14}$ \\ \hline
 $256$ &  $ 1.4\times 10^{-14}$ & $1.3\times 10^{-12}$ & $1.2\times 10^{-13}$ & $1.9\times 10^{-13}$ \\ \hline
 $512$ & $5.2\times 10^{-14}$ & $4.6\times 10^{-11}$ & $1.0\times 10^{-12}$ & $4.9\times 10^{-12}$ \\ \hline
 $1024$  & $1.2\times 10^{-13}$ & $1.0\times 10^{-09}$ & $1.2\times 10^{-11}$ & $1.0\times 10^{-10}$ \\ \hline
    \end{tabular}
  \end{center}
\end{table}



\begin{table}[ht]
  \caption{Relative residual norms:
 GENP with Gaussian random multipliers}
  \label{tab44g}
  \begin{center}
    \begin{tabular}{| c | c | c | c | c | c | c |}
      \hline
      \bf{dimension} & \bf{iterations} & \bf{min} & \bf{max} & \bf{mean} & \bf{std} \\ \hline
 $64$ & $0$ & $3.6\times 10^{-13}$ & $1.4\times 10^{-11}$ & $5.3\times 10^{-12}$ & $8.4\times 10^{-12}$ \\ \hline
 $64$ & $1$ & $7.3\times 10^{-15}$ & $2.8\times 10^{-13}$ & $3.2\times 10^{-14}$ & $6.4\times 10^{-14}$ \\ \hline
 $256$ & $0$ & $5.1\times 10^{-12}$ & $3.8\times 10^{-9}$ & $9.2\times 10^{-10}$ & $4.7\times 10^{-11}$ \\ \hline
 $256$ & $1$ & $4.8\times 10^{-15}$ & $9.2\times 10^{-10}$ & $8.6\times 10^{-11}$ & $2.1\times 10^{-11}$ \\ \hline
    \end{tabular}
  \end{center}
\end{table}




\begin{table}[ht]
  \caption{Relative residual norms:
 GENP with Gaussian circulant random  multipliers 
}
  \label{tab44}
  \begin{center}
    \begin{tabular}{| c | c | c | c | c | c | c |}
      \hline
      \bf{dimension} & \bf{iterations} & \bf{min} & \bf{max} & \bf{mean} & \bf{std} \\ \hline
 $64$ & $0$ & $4.7\times 10^{-14}$ & $8.0\times 10^{-11}$ & $4.0\times 10^{-12}$ & $1.1\times 10^{-11}$ \\ \hline
 $64$ & $1$ & $1.9\times 10^{-15}$ & $5.3\times 10^{-13}$ & $2.3\times 10^{-14}$ & $5.4\times 10^{-14}$ \\ \hline
 $256$ & $0$ & $1.7\times 10^{-12}$ & $1.4\times 10^{-7}$ & $2.0\times 10^{-9}$ & $1.5\times 10^{-8}$ \\ \hline
 $256$ & $1$ & $8.3\times 10^{-15}$ & $4.3\times 10^{-10}$ & $4.5\times 10^{-12}$ & $4.3\times 10^{-11}$ \\ \hline
 $1024$ & $0$ & $1.7\times 10^{-10}$ & $4.4\times 10^{-9}$ & $1.4\times 10^{-9}$ & $2.1\times 10^{-9}$ \\ \hline
 $1024$ & $1$ & $3.4\times 10^{-14}$ & $9.9\times 10^{-14}$ & $6.8\times 10^{-14}$ & $2.7\times 10^{-14}$ \\ \hline
    \end{tabular}
  \end{center}
\end{table}



\section{Conclusions}\label{sconcl}


It is well known that Gaussian 
 (that is standard Gaussian random) matrices 
tend to be well conditioned,
and  
this property motivates our preprocessing of
 well conditioned nonsingular  input matrices  
with Gaussian multipliers to support the application of 
GENP and block
 Gaussian elimination.  
 Both of these algorithms can fail in 
practical numerical computations without 
preprocessing, but we prove that  with Gaussian multipliers
the algorithms are expected to be locally safe,
which would achieve the purpose of pivoting. Namely the 
absolute values of the reciprocals of
all pivot elements of GENP  and 
the norms of the inverses of
all pivot blocks of block Gaussian elimination
 are likely to be reasonably bounded.
Our tests were in good accordance with that formal study.
We generated matrices that 
were hard for GENP, but the problems were consistently 
avoided where we preprocessed the inputs with Gaussian multipliers.
Moreover empirically we obtained essentially the same results even 
where we used
 circulant random  multipliers.
Their choice accelerates multiplication significantly, and
particularly dramatically in the important case where 
the input matrix has Toeplitz structure. That choice also limits
randomization  to $n$
random parameters for an $n\times n$
input.
 Formal support for the 
empirical power of these multipliers is a natural research challenge.
\bigskip




{\bf Acknowledgements:}
Our research has been supported by NSF Grant CCF--1116736 and
PSC CUNY Awards 64512--0042 and 65792--0043.







\bigskip


{\bf {\LARGE {Appendix}}}
\appendix


\section{On the algebraic variety of low-rank matrices}\label{savrnk}


The following simple result (not used in this paper)
shows that the $m\times n$  matrices of a rank $\rho$
form an algebraic variety of the dimension $d_{\rho}=(m+n-\rho)\rho$
in the space $\mathbb  R^{m\times n}$, and clearly $d_{\rho}<mn$
for $\rho<\min\{m,n\}$. 


\begin{fact}\label{far1}
The set $\mathbb A$ of $m\times n$ matrices of rank  $\rho$
is an algebraic variety of dimension  $(m+n-\rho)\rho$.
\end{fact}
\begin{proof}
Let $A$ be an $m\times n$ matrix of a rank $\rho$
with a nonsingular leading $\rho\times \rho$ block $B$
and write $A=\begin{pmatrix}
B   &   C  \\
D  &   E
\end{pmatrix}$.
Then the $(m-\rho)\times (n-\rho)$ {\em Schur complement} $E-DB^{-1}C$
must vanish, which imposes $(m-\rho)(n-\rho)$ algebraic 
equations on the entries of the matrix $A$. 
Similar argument can be applied  where any $\rho\times \rho$
submatrix of the matrix $A$ 
(among $\begin{pmatrix}
m      \\
\rho
\end{pmatrix}\begin{pmatrix}
n     \\
\rho
\end{pmatrix}$ such submatrices)
is nonsingular. Therefore 
$\dim \mathbb A=mn-(m-\rho)(n-\rho)=(m+n-\rho)\rho$.
\end{proof}








\section{Uniform random sampling and nonsingularity of random matrices}\label{srsnrm}


{\em Uniform random sampling} of elements from a finite set $\Delta$ is their selection   
from  this set at random, independently of each other and
under the uniform probability distribution on the set $\Delta$. 


The total degree of a multivariate monomial is the sum of its degrees
in all its variables. The total degree of a polynomial is the maximal total degree of 
its monomials.


\begin{lemma}\label{ledl} \cite{DL78}, \cite{S80}, \cite{Z79}.
For a set $\Delta$ of a cardinality $|\Delta|$ in any fixed ring  
let a polynomial in $m$ variables have a total degree $d$ and let it not vanish 
identically on the set $\Delta^m$. Then the polynomial vanishes in at most 
$d|\Delta|^{m-1}$ points of this set. 
\end{lemma}

\begin{theorem}\label{thdl} 
Under the assumptions of Lemma \ref{ledl} let the values of the variables 
of the polynomial be randomly and uniformly sampled from a finite set $\Delta$. 
Then the polynomial vanishes with a probability at most $\frac{d}{|\Delta|}$. 
\end{theorem}


\begin{corollary}\label{codlstr} 
Let the entries of a general or Toeplitz  $m\times n$ 
matrix have been randomly and uniformly 
sampled from a finite set $\Delta$ of cardinality $|\Delta|$ (in any fixed ring). 
Let $l=\min\{m,n\}$.
Then (a) every $k\times k$ submatrix $M$ for $k\le l$ is nonsingular with a probability at 
least $1-\frac{k}{|\Delta|}$ and (b) is strongly nonsingular with a probability at least 
$1-\sum_{i=1}^k\frac{i}{|\Delta|}= 1-\frac{(k+1)k}{2|\Delta|}$.
\end{corollary}


\begin{proof}
The claimed properties of nonsingularity and nonvanishing hold for ge\-ne\-ric matrices. 
The singularity of a $k\times k$ matrix means that its determinant vanishes,
but the determinant is a polynomial of total degree $k$ in the entries. Therefore
Theorem \ref{thdl} implies
parts (a) and consequently (b). Part (c) follows because a fixed entry of the inverse vanishes
if and only if the respective entry of the adjoint vanishes, but up to the sign the latter 
entry is the determinant of a $(k-1)\times (k-1)$ submatrix of the input matrix $M$, and so it is
a polynomial of degree $k-1$ in its entries. 
\end{proof}


\end{document}